\documentclass{amsart}

\usepackage{hyperref}
\hypersetup{nesting=true,debug=true,naturalnames=true}
\usepackage{graphicx,amssymb,upref,units}

\usepackage{tikz}

\let\<\langle
\let\>\rangle

\let\uml\"

\usepackage{amsmath,amssymb,,xy,mathrsfs,amsthm,xcolor,amsxtra,graphicx,verbatim}
\usepackage{hyperref}
\hypersetup{colorlinks=true,urlcolor=magenta,citecolor=magenta,linkcolor=blue}
\usepackage{enumerate}
\usepackage{colonequals}
\usepackage{multirow}

\newcommand{\far}[2]{\ensuremath{( \hspace{-2pt} ( #1 ) \hspace{-2pt} )_{\ \hspace{-4pt}_#2}}}
\usepackage[mathcal]{euscript}
\newcommand{\Gal}{\operatorname{Gal}}

\newcommand{\Fq}{\mathbf{F}_{q}}

\newcommand{\Z}{\mathbf{Z}}

\newcommand{\R}{\operatorname{\mathbf{R}}}

\newcommand{\Q}{\mathbf{Q}}

\newcommand{\C}{\mathbf{C}}

\newcommand{\ddef}{\colonequals}

\newcommand{\K}{\mathcal{K}}
\newcommand{\KK}{\mathsf{K}}
\newcommand{\NP}{\operatorname{NP}}

\numberwithin{equation}{subsection}

\theoremstyle{plain}
\newtheorem{thm}[equation]{Theorem}
\newtheorem{conjecture}[equation]{Conjecture}

\newtheorem{lem}[equation]{Lemma}

\newtheorem{cor}[equation]{Corollary}
\newtheorem{prop}[equation]{Proposition}

\theoremstyle{remark}

\newtheorem{exm}[equation]{Example}

 \input xy
 \xyoption{all}

\begin{document}

\title{On the Irreducibility of the Krawtchouck Polynomials}

\author{John Cullinan}
\address{Department of Mathematics, Bard College, Annandale-On-Hudson, NY 12504, USA}
\email{cullinan@bard.edu}
\urladdr{\url{http://faculty.bard.edu/cullinan/}}



\begin{abstract}
The Krawtchouck polynomials arise naturally in both coding theory and probability theory and have been studied extensively from these points of view.  However, very little is known about their irreducibility and Galois properties.  Just like many classical families of orthogonal polynomials (\emph{e.g.} the Legendre and Laguerre), the Krawtchouck polynomials can be viewed as special cases of Jacobi polynomials.  In this paper we determine the Newton Polygons of certain Krawtchouck polynomials and show that they are very similar to those of the Legendre polynomials (and exhibit new cases of irreducibility).  However, we also show that their Galois groups are significantly more complicated to study, due to the nature of their coefficients, versus those of other classical orthogonal families.
\end{abstract}

\maketitle

\section{Introduction}
\subsection{Background} Let $q$ be a prime power and $n$ a positive integer.  The \emph{Krawtchouck polynomials}, defined explicitly as
\begin{align} \label{initial_defn}
K_k(n,q;x) = \sum_{j=0}^k(-q)^j(q-1)^{k-j} \binom{n-j}{k-j} \binom{x}{j},
\end{align}
arise naturally in  coding theory  over $\Fq^n$ \cite[\S1.2]{van_lint}; see also \cite[\S2.82]{szego} for a more general development. We immediately set $q=2$ for the remainder of the paper.  Because the Krawtchouck polynomials are related to the \emph{Jacobi polynomials}:
\begin{align} \label{jac_defn}
P_n^{(\alpha,\beta)}(x) = \sum_{j=0}^n \binom{n+\alpha}{n-j}\binom{n+\beta}{j} \left(\frac{x-1}{2}\right)^j \left( \frac{x+1}{2} \right)^{n-j},
\end{align}
and there are competing uses of the index ``$n$'',  we will change our notation in (\ref{initial_defn}) to be more in line with that of (\ref{jac_defn}), as follows.  Let $\Q(t)$ be a function field of transcendence degree 1 over the rational numbers $\Q$ and let $n \in \Z_{\geq 0}$.  In the polynomial ring $\Q(t)[x]$ we now set
\[
\K_n^{(t)}(x) \ddef K_n(t,2;x).
\]
By direct substitution one easily checks that 
\[
\K_n^{(t)}(x) = 2^n P_n^{(t-x-n,x-n)}(0).
\]
The arithmetic of the polynomials $\K_n^{(t)}(x)$ and their specializations for $t \in \Q$ are the primary objects of study in this paper.   Note that from the point of view of Coding Theory, $t$ is the dimension of the ambient space $\Fq^t$; for those applications one would choose $t \in \Z_{>0}$.

There are many well-known families of orthogonal polynomials that are specializations of the Jacobi polynomials.  For example, the Gegenbauer polynomials $C_n^{(\lambda)}(x)$ are given by \cite[(4.71)]{szego}
\[
C_n^{(\lambda)}(x) = \frac{\Gamma(\lambda + 1/2)}{\Gamma(2\lambda)} \, \frac{\Gamma(n+2\lambda)}{\Gamma(n+\lambda+1/2)}\, P_n^{(\lambda-1/2,\lambda-1/2)}(x),
\]
and the Legendre polynomials $P_n(x)$ are given by $P_n(x) = P_n^{(0,0)}(x)$ \cite[(4.7.2)]{szego}. By specializing ``at infinity'', one can also view the Generalized Laguerre polynomials ${L}_n^{(\alpha)}(x)$ as special cases of the  of the Jacobi polynomials \cite[(5.3.3)]{szego}:
\[
L_n^{(\alpha)}(x) = \lim_{\beta \to \infty} P_n^{(\alpha,\beta)}(1-2x/\beta).
\]
We also point out that the Hermite polynomials are, in turn, specializations of the ${L}_n^{(\alpha)}(x)$ \cite[(5.6.1)]{szego}.

The Krawtchouck polynomials do not arise by specializing the parameters $\alpha$ and $\beta$ of the Jacobi polynomials, but rather as focusing on the parameter as the ``variable''.  To the best of our knowledge, the Krawtchouck polynomials have not been extensively studied for their irreducibility and Galois properties.  And even though they are not typically viewed as a subfamily of the Jacobi polynomials, the Krawtchouck polynomials remarkably still exhibit similar arithmetical properties to classical specializations (such as the Legendre polynomials).  But at the same time, some of the standard arithmetical techniques do not immediately apply to these polynomials as they do to the classical families. 

A property that many Jacobi polynomials share that when the degree is a prime power (or close to a prime power) the Newton Polygons at that prime exhibit interesting behavior, allowing one to deduce irreducibility and Galois-theoretic information.  See \cite{ch} and \cite{wahab}  for results of this type for the Legendre polynomials and \cite{rylan} and \cite{kz}  for a family of lifts of the supersingular polynomials (originally due to Kaneko and Zagier \cite{kzorig}), with applications to elliptic curves.  Even though in some special cases we are able to prove that the Legendre polynomials and the supersinglar lifts are irreducible with large Galois group, it remains an open problem to determine the irreducibility and Galois groups for all of them.  For the Galois groups of the Generalized Laguerre Polynomials, see \cite{filaseta} and \cite{farshid} as well for complementary analyses. 

Returning to the Krawtchouck polynomials, many of the same techniques that work for classical specializations of the Jacobi polynomials simply do not bear fruit in this case.  The Krawtchouck polynomias do not satisfy a \emph{Schur factorization} \cite[Thm.~6.1]{ch}, nor do their discriminants have a predictable formula like the classical Jacobi polynomials do (as in \cite[(4)]{kz}) since we are interchanging the roles of ``parameter'' and ``variable''.  Nonetheless, some of their specializations exhibit enough interesting arithmetic information to merit further study.

\subsection{Main Results}  With all notation as above, we can now state the main results of the paper, starting with the $p$-adic Newton Polygon of $\K_n^{(t)}(x)$; we restrict to the case of $p=2$. Write the degree $n$ in base 2 as
\[
n = 2^{j_1} + 2^{j_2}+\cdots + 2^{j_k},
\]
where $0\leq j_1 < j_2 < \cdots < j_k$.  We say that a rational specialization $\K_n^{(t_0)}(x)$ has a \emph{degree-based} 2-adic Newton Polygon if $\NP_2(\K_n^{(t_0)}(x))$ has segments of length $2^{j_1},\cdots,2^{j_k}$ with slopes $-2^{-j_1},\dots,-2^{-j_k}$, respectively.

\begin{exm} \label{exm1}
Writing the polynomial $\K_{19}^{(19)}(x) = \sum_{j=0}^{19} a_jx^j$, we compute $(j,v_2(a_{n-j}))$ for $j=0,\dots,19$:
\begin{align*}
&\textbf{(0, 3)},
\textbf{(1, 2)},
(2, 3),
\textbf{(3, 1)},
(4, 4),
(5, 3),
(6, 4),
(7, 3),
(8, 3),
(9, 2),
(10, 3),\\
&(11, 5),
(12, 6),
(13, 6),
(14, 7),
(15, 12),
(16, 11),
(17, 14),
(18, 18),
\textbf{(19, 0)},
\end{align*}
yielding the 2-adic Newton Polygon
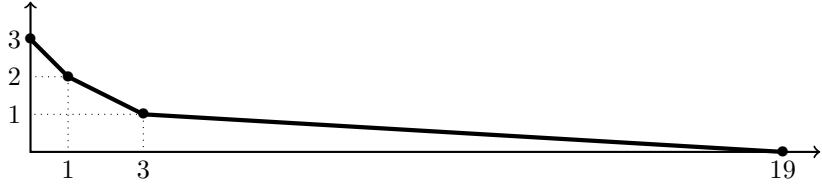
\begin{center}
\begin{tikzpicture}[scale=1,sizefont/.style={scale = 2}]
        \draw [<->,thick] (0,2) node (yaxis) [above] {}
        |- (10.5,0) node (xaxis) [right] {};
   \draw[ultra thick] (0,3/2)   --  (1/2,1); 
   \draw[ultra thick] (1/2,1)   --  (3/2,1/2); 
   \draw[ultra thick] (3/2,1/2)   --  (10,0); 
   \draw[dotted] (1/2,1)   --  (1/2,0);
   \draw[dotted] (3/2,1/2)   --  (3/2,0);
   \draw[dotted] (1/2,1)   --  (0,1);
   \draw[dotted] (3/2,1/2)   --  (0,1/2);  

  \draw (0,1/2) node[left] {$1$};
  \draw (0,1) node[left] {$2$};
  \draw (0,3/2) node[left] {$3$};
  \draw (1/2,0) node[below] {$1$};
  \draw (3/2,0) node[below] {$3$};
\draw (1/2,1) node {$\bullet$};
\draw (10,0) node {$\bullet$};
\draw (3/2,1/2) node {$\bullet$};
  \draw (10,0) node[below] {$19$};
\draw (0,3/2) node {$\bullet$};

\end{tikzpicture}
\end{center}
which we verify is degree-based: $19 = 2^0 + 2^1 + 2^4$.
\end{exm}

Our first result determines certain specializations that have degree-based 2-adic Newton Polygons.  

\begin{thm} \label{npthm}
Let $n\geq1$ and write $n=2^{j_1} + 2^{j_2}+\cdots + 2^{j_k}$ with $0\leq j_1 < j_2 < \cdots < j_k$.  Then  for all $t \in [n,n+2^{j_1}) \cap \Z$, the 2-adic Newton Polygon of $\K_n^{(t)}(x)$ is degree-based.
\end{thm}

This is an identical result to \cite[Thm.~3.1]{wahab} for the shifted Legendre polynomials $P_n(2x+1)$.  The same {proof} does not carry over, however, due to the different nature of the coefficients. Among other things, the Krawtchouck polynomials do not have a discriminant formula resembling that of the Jacobi or Legendre polynomials (see \cite[Lem.~2.1]{ch}), due to the fact that the derivatives of the Krawtchouck polynomials are not themselves members of the same family.  One of the key ingredients of Wahab's proof for the Legendre polynomials is a precise formula for, and beautiful factorization of,  the numbers 
\[
\frac{d^r (P_n(x))}{dx^r} \bigg|_{x=1}
\]
due to Grosswald \cite{grosswald}.

\medskip

By taking $n = 2^k$, we obtain new families of irreducible specializations as an immediate corollary.

\begin{cor} \label{irred_cor}
Let $k\geq 0$, $n = 2^k$, and $t \in [2^k,2^{k+1}) \cap \Z$.  Then $\K_n^{(t)}(x)$ is Eisenstein at $p=2$ and hence is irreducible.
\end{cor}

Turning to the Galois groups, we observe that the shifted polynomial $\K_n^{(t)}(x+t/2)$ is an even function when $n$ is even and odd when $n$ is odd.  Write $n = 2m+\delta$, for $\delta \in \lbrace 0,1\rbrace$ and consider the polynomial $\KK_m^{(\delta,t)}(x)$, where 
\begin{align} \label{underlying}
\KK_m^{(1,t)}(x^2) &= \K_{2m+1}^{(t)}(x+t/2)/x \\
\KK_m^{(0,t)}(x^2) &= \K_{2m}^{(t)}(x+t/2).
\end{align}
Thus, when $\KK_m^{(\delta,t)}(x^2)$ is irreducible over $\Q$, it must be the case that 
\[
\Gal_\Q (\KK_m^{(\delta,t)}(x^2)) \subseteq S_2 \wr S_m,
\]
and 
\[
{\rm E}_m^{(\delta,t)} \ddef \Gal_\Q (\KK_m^{(\delta,t)}(x^2)) / \Gal_\Q (\KK_m^{(\delta,t)}(x))
\]
is an elementary 2-group of rank at most $m$. Determining the groups $\Gal_\Q (\KK_m^{(\delta,t)}(x))$ and ${\rm E}_m^{(\delta,t)}$ are very different types of calculations; see \cite[\S 4.1]{ch} for an analagous setup with the Legendre Polynomials.   Our goal is to introduce the following conjecture into the literature.

\begin{conjecture} \label{conj}
Fix $n \geq 1$ and write $n=2m+\delta$.  Then for all $t \in \Q$ such that $\KK_m^{(t,\delta)}(x)$ is irreducible, we have $\Gal_\Q (\KK_m^{(t,\delta)}(x)) \simeq S_m$. 
\end{conjecture}

A key consequence of this conjecture is that (unlike many other families of polynomials, such as the Generalized Laguerre polynomials) it implies whenever $\KK_m^{(t,\delta)}(x)$ is irreducible, its discriminant is never a square.

In contrast to the Legendre polynomials or the supersingular lifts of Kaneko and Zagier (see \cite{rylan} and \cite{ch}), we can prove almost nothing in general about the $\Gal_\Q (\KK_m^{(t,\delta)}(x))$; see Section \ref{comp} below for a discussion of this and some data.

Finally, despite not addressing the groups ${\rm E}_m^{(\delta,t)}$ in generality, we present the following Proposition which shows that not all specializations yield rank-$m$ 2-groups, hence there are interesting arithmetical constraints to be explored in the future.

\begin{prop} \label{-1prop}
Let $k\geq 2$.  Then the polynomials $\K_{2^k}^{(-1)}(x)$ are irreducible and the group ${\rm E}_{2^{k}}^{(0,-1)}$ does not have maximal 2-rank, $2^{k-1}$. 
\end{prop}

Conjecture \ref{conj} and Proposition \ref{-1prop} together suggest that for all specializations $t_0 \in \Q$, then Galois group $\Gal_\Q (\KK_n^{(t_0,\delta)}(x^2))$ surjects onto $S_n$, but the kernel may vary in size as $t_0$ varies over the irreducible specializations.

\section{Preliminaries and Setup}

Retaining the notation of the Introduction, we start by proving our observation on the shifted polynomials $\K_n^{(t)}(x+t/2)$.

\begin{lem} \label{evenoddlem}
If $n$ is even then $\K_n^{(t)}(x+t/2)$ is an even function and if $n$ is odd then $\K_n^{(t)}(x+t/2)$ is an odd function.
\end{lem}

\begin{proof}
We are content to give a proof of a single special case; the general result follows similarly, but is more notationally cumbersome.  Let $n$ be even and set $t=n$.  Then 
\begin{align*}
K_n^{(n)}(x) &= \sum_{j=0}^n (-2)^j  \binom{x+n/2}{j} \\
&= \sum_{j=0}^n \frac{(-2)^j}{j!} \prod_{k=0}^{j-1} (x + n/2 - k) \\
&= \sum_{j=0}^n \frac{(-1)^j}{j!} \prod_{k=0}^{j-1} (2x + n - 2k).
\end{align*}
Let $F_n(u) = \sum_{j=0}^n \frac{(-1)^j}{j!} \prod_{k=0}^{j-1} (u + n - 2k)$.  We will show that $F_n(u)$ is an even function for all even $n$ by induction on the degree, where the base case $F_0(u) =1$.  

Let $n > 0$ and consider $F_{n+2}(u)$.  We start by unwinding 
\[
u+n+2 = (u+n-2j+2)+2j
\] 
so that the product
\[
\prod_{k=0}^{j-1} (u+n+2-2k)
\]
can be written as the sum
\[
(u+n) \cdots  (u+n-2j+2) + 2j(u+n) \cdots (u+n - 2j+4).
\]
If we first separate  $F_{n+2}(u)$ as follows,
\begin{align*}
F_{n+2}(u) &= \sum_{j=0}^{n+1} \frac{(-1)^j}{j!} \prod_{k=0}^{j-1} (u + n + 2 - 2k)  + \frac{(-1)^{n+2}}{(n+2)!} \prod_{k=0}^{n+1} (u+n+2-2k),
\end{align*}
then the sum can be rewritten as
\begin{align*}
\sum_{j=0}^{n+1} \frac{(-1)^j}{j!} \prod_{k=0}^{j-1} (u + n + 2 - 2k) &= \sum_{j=0}^{n+1} \frac{(-1)^j}{j!} \prod_{k=0}^{j-1} (u + n  - 2k) + \sum_{j=0}^{n+1} \frac{(-1)^j}{j!} 2j \prod_{k=0}^{j-2} (u + n - 2k).
\end{align*}
In turn, we have
\begin{align*}
\sum_{j=0}^{n+1} \frac{(-1)^j}{j!} \prod_{k=0}^{j-1} (u + n  - 2k) &= F_n(u) + \frac{(-1)^{n+1}}{(n+1)!} \prod_{k=0}^{n} (u + n  - 2k),\text{ and}  \\
\sum_{j=0}^{n+1} \frac{(-1)^j}{j!} 2j \prod_{k=0}^{j-2} (u + n - 2k) &= -2F_n(u).
\end{align*}
By induction, the polynomial $F_n(u)$ is an even function.  Hence it remains to show that the remaining terms,
\begin{align}\label{last}
\frac{(-1)^{n+1}}{(n+1)!} \prod_{k=0}^{n} (u + n  - 2k) + \frac{(-1)^{n+2}}{(n+2)!} \prod_{k=0}^{n+1} (u + n +2 - 2k),
\end{align}
are even. But 
\begin{align*}
\frac{(-1)^{n+2}}{(n+2)!} \prod_{k=0}^{n+1} (u + n +2 - 2k) &= \frac{(-1)^{n+2}}{(n+2)!} \left[ [u+n+2] \prod_{k=0}^{n} (u + n +2k) \right] \\
&=\frac{(-1)^{n+2}u}{(n+2)!} \prod_{k=0}^{n} (u + n +2k) + \frac{(-1)^{n+2}}{(n+1)!} \prod_{k=0}^{n} (u + n +2k).
\end{align*}
Substituting into (\ref{last}), we are left to consider 
\[
\frac{(-1)^{n+2}u}{(n+2)!} \prod_{k=0}^{n} (u + n +2k) = \frac{(-1)^{n+2}}{(n+2)!} (u+n)\cdots(u+2)u^2(u-2)\cdots(u-n),
\]
which is clearly even. 
\end{proof}

As explained above, this puts a constraint on the Galois group of $\K_n^{(t)}(x)$, suggesting that we divide our study of the Galois theory of $\K_n^{(t)}(x)$ into two parts: the Galois group of $\KK_m^{(\delta,t)}(x)$ as in (\ref{underlying}) and the group ${\rm E}_m^{(\delta,t)}$.  

Before turning to the proofs of our main results we set our convention for the $p$-adic Newton Polygon.  If $p$ is a prime number and $f \in \Q[x]$ is a polynomial written explicitly as $f(x) = \sum_{j=0}^n a_j x^j$, then the $p$-adic Newton Polygon of $f$, denoted $\NP_p(f)$ is the lower convex hull of the points
\begin{align} \label{np_def}
(n-j,v_p(a_j)) \in \Z \times \Z.
\end{align}
The \emph{breaks} of $\NP_p(f)$ are the coordinates (\ref{np_def}) where there is a change of slope.  For example, the breaks of the Newton Polygon in Example \ref{exm1} occur at $(0,3)$, $(1,2)$, $(3,1)$, and $(19,0)$.  

In the course of our proofs we will need the $p$-adic valuation of the factorial; this is well-known.  For an integer $n$, let it's base-$p$ expansion be given by 
\[
n = \sum_{j=0}^k c_j p^j
\]
for integers $0 \leq c_j \leq p-1$.  Then 
\begin{align} \label{ordfactorial}
v_p(n!) = \frac{n-\sum_{j=0}^kc_j}{p-1}.
\end{align}

\section{Proofs}

With all of the notation and setup complete, we turn to the proofs of results described in the Introduction, starting with the shape of the 2-adic Newton Polygon at certain specializations. 

\begin{proof}[Proof of Theorem \ref{npthm}]
Our proof will proceed in several steps.  First, we will consider the polynomials $\K_n^{(n)}(x) = \sum_{j=0}^na_jx^j$ and show that for these polynomials, we have
\[
v_2 (a_{n - \sum_{\ell =1}^r 2^{j_\ell}}) = k-r.
\]
This will establish $k+1$ distinguished pairs: $(\sum_{\ell =1}^r2^{j_\ell},k-r)$ for $r=0,\dots,k$.  We next show that these are the breaks of the 2-adic Newton Polygon by showing that the valuations of the intermediate coefficients lie above those of the $k+1$ distinguished coefficients.  Finally, for specializations $t \in [n,n+2^{j_1})$, we show that the breaks of the Newton Polygons occur at the same places.

Write $\K_n^{(n)}(x) = \sum_{j=0}^n (-2)^j \binom{x}{j} = \sum_{m=0}^n a_mx^m$ with 
\[
\binom{x}{j} = \frac{1}{j!}\sum_{k=1}^j b_k^{j} x^k.
\]
The superscript ``$j$'' is understood to be an index, not an exponent.  Substituting this into the expression for $\K_n^{(n)}(x)$, we get that
\begin{align} \label{am}
a_m = \frac{(-2)^m}{m!} \left( 1 + \frac{(-2)b_m^{m+1}}{m+1} + \frac{(-2)^2b_{m}^{m+2}}{(m+1)(m+2)} + \cdots + \frac{(-2)^{n-m}b_{m}^{n}}{(m+1)(m+2)\cdots n} \right).
\end{align}

We now set $m = n - \sum_{\ell =1}^r 2^{j_\ell}$.  Observe that $m$ is even (so that $m+1$ is odd), $b_m^{s} \in \Z$ for all $s \geq m+1$, and that 
\[
v_2 \left( \frac{(-2)^{s-m}}{(m+1) \cdots s} \right) \geq 1
\]
for all $s \geq m+1$.  Thus, 
\[
v_2 \left( 1 + \frac{(-2)b_m^{m+1}}{m+1} + \frac{(-2)^2b_{m}^{m+2}}{(m+1)(m+2)} + \cdots + \frac{(-2)^{n-m}b_{m}^{n}}{(m+1)(m+2)\cdots n} \right) = 0,
\]
hence $v_2(a_m) = m - v_2(m!)$. Now apply (\ref{ordfactorial}) to get
\[
v_{2} \left(a_{n - \sum_{\ell =1}^r 2^{j_\ell}} \right) = {n - \sum_{\ell =1}^r 2^{j_\ell}} - \left( n - \sum_{\ell =1}^r 2^{j_\ell} - (k-r) \right) = k-r.
\]
This produces our $k+1$ distinguished points and completes the first part of the proof.

\medskip

Next we turn to the intermediate coefficients.  By our work above, the distinguished points occur at the coordinates 
\[
\left( \sum_{\ell =1}^r2^{j_\ell},k-r \right)
\]
for $r=0,\dots,k$.  Fix an $x$-coordinate $u \in [0,n] \cap \Z$ that is not one of the $\sum_{\ell =1}^r2^{j_\ell}$.  Then 
\[
2^{j_1} + \cdots + 2^{j_r} < u < 2^{j_1} + \cdots + 2^{j_{r+1}},
\]
so that 
\[
2^{j_{r+1}} + \cdots + 2^{j_k} > n-u > 2^{j_{r+2}} + \cdots + 2^{j_k}.
\]
Thus the 2-adic expansion of $n-u$ is
\[
n-u = 2^{j_k} + \cdots 2^{j_{r+2}} + 2^{s_1} + \cdots 2^{s_h},
\]
with $h \geq 1$, and so the number of 2-adic digits of $n-u$ is 
\[
k-(r+2)+1+h \geq k-r.
\]
Therefore, for all such $u$, the 2-adic valuation of (\ref{am}) lies strictly above the segment connecting the points
\[
\left( \sum_{\ell =1}^r2^{j_\ell},k-r \right) \text{ and } \left( \sum_{\ell =1}^{r+1} 2^{j_\ell},k-r-1 \right).
\]
This establishes that the distinguished points are the breaks of Newton Polygon of $\K_n^{(n)}(x)$. 

\medskip

Finally, let $t \in [n,n+2^{j_1})$.  We will show that the 2-adic Newton Polygon of $\K_n^{(t)}(x)$ is identical to that of $\K_n^{(n)}(x)$ by showing that the breaks are the same.  Since $\K_n^{(t)}(x)  = \sum_{j=0}^n (-2)^j \binom{t-j}{n-j} \binom{x}{j}$, we can view $\K_n^{(t)}(x)$ as a twist of $\K_n^{(n)}(x)$:
\[
\K_n^{(t)}(x) = \sum_{m=0}^n \binom{t-m}{n-m} a_mx^m,
\]
where the $a_m$ are the same as in (\ref{am}). Writing $t = n + \epsilon$ so that
\[
\binom{t-m}{n-m} = \frac{(\epsilon + 1)(\epsilon+2) \cdots (\epsilon + n-m)}{1\cdot 2 \cdots (n-m)}.
\]
since $n = \sum_{\ell=1}^k 2^{j_\ell}$, as long as $0 \leq \epsilon < 2^{j_1}$ and $m=\sum_{\ell=1}^r 2^{j_\ell}$ the numerator and denominator will have exactly the same 2-valuation, \emph{i.e.}, that
\[
(\sum_{\ell =1}^r2^{j_\ell},k-r) 
\]
are vertices of the Newton Polygon for $r=0,\dots,k$.  And since the $\binom{t-m}{n-m}$ are always integers, the valuations at the intermediate coefficients are at least as large as when $t=n$, we conclude that the Newton Polygons are identical, as claimed. 
 \end{proof}

As observed in the Introduction, Theorem \ref{npthm} immediately implies the following Corollary.

\begin{cor}
Let $k\geq 0$, $n = 2^k$, and $t \in [2^k,2^{k+1}) \cap \Z$.  Then $\K_n^{(t)}(x)$ is Eisenstein at $p=2$ and hence is irreducible.
\end{cor}

\begin{proof}
If $n=2^k$, then a degree-based 2-adic Newton Polygon consists of a single segment.  Writing $\K_n^{(t)}(x) = \sum a_jx^k$, observe that for all $t\in [2^k,2^{k+1})$ we have $v_2(a_n) = 1$ and $v_2(a_0)=0$, so that this segment does not pass through any other integer points in the plane.  Hence $\K_n^{(t)}(x)$ is Eisenstein at $p=2$.
\end{proof}

Now we prove Proposition \ref{-1prop}.

\begin{proof}[Proof of Proposition \ref{-1prop}]
First we prove that the polynomials $\K_{2^k}^{(-1)}(x)$ are irreducible.  Observe that for any $n>0$, the 
leading coefficient of $\K_n^{(t)}(x)$ is $(-2)^n/n!$ (independent of $t$), and the constant coefficient is $\binom{t}{n}$. If $t=-1$, then $\binom{-1}{n} = (-1)^n$.  If, in addition, $n=2^k$ is a power of 2, then 
\[
v_2 \left( \frac{(-2)^{2^k}}{2^k!}\right) = 1.
\]
The intermediate coefficients are all even (since $(-2)^j$ has larger 2-valuation than $(2^k - j)!$) and so $\K_{2^k}^{(-1)}(x)$ is Eisenstein at 2 and hence irreducible. 

Turning to the Galois group, all of the roots of $\KK_{2^{k-1}}^{(0,-1)}(x)$ are negative real numbers by Descartes' rule of signs, hence all of the roots of $\K_{2^k}^{(-1)}(x-1/2)$ are purely imaginary.  Thus, we can list the roots of $\KK_{2^{k-1}}^{(0,-1)}(x)$ as
\[
\lbrace \xi_1,\dots,\xi_{2^{k-1}} \rbrace \subseteq \R_{<0},
\]
and those of $\K_{2^k}^{(-1)}(x-1/2)$ as
\[
\lbrace \theta_1,\overline{\theta}_1,\dots,\theta_{2^{k-1}}, \overline{\theta}_{2^{k-1}} \rbrace \subseteq \C \setminus \R,
\]
with the convention that $\theta_i\overline{\theta}_i = \xi_i$.

Let ${\rm L} = \Q(\xi_1,\dots,\xi_{2^{k-1}})$ be the splitting field of $\KK_{2^{k-1}}^{(0,-1)}(x)$ with Galois group $\Gal({\rm L}/\Q)$. If ${\rm M}$ is the splitting field of $\K_{2^{k}}^{(-1)}(x)$ over $\Q$, then ${\rm M}$ is generated over ${\rm L}$ by the elements $\lbrace \theta_1,\overline{\theta}_1,\dots,\theta_{2^{k-1}}, \overline{\theta}_{2^{k-1}} \rbrace$, and $\Gal({\rm M}/{\rm L})$ is an elementary 2-group of rank at most $2^{k-1}$.  We will now show that the rank is $< 2^{k-1}$.

To do this, we translate back and consider the polynomial $\K_{2^k}^{(-1)}(x)$, with roots
\[
\eta_i:=-1/2 + \theta_i, \text{ and } \overline{\eta}_i,
\]
for $i = 1,\dots, 2^{k-1}$. We note that since this is a linear shift in the polynomial, all of the splitting fields of the shifts are identical to those of the original polynomials.  Since the degree of $\K_{2^k}^{(-1)}(x)$ is even, its constant coefficient is $\binom{-1}{n}=1$, hence 
\[
\prod_{i=1}^{2^{k-1}} \eta_i\overline{\eta}_i = \underbrace{\prod_{i=1}^{2^{k-1}} \eta_i}_{\lambda} \underbrace{\overline{\prod_{i=1}^{2^{k-1}}\eta_i}}_{\overline{\lambda}} = 1.
\]
It follows that the fields ${\rm L}(\lambda)$ and ${\rm L}(\overline{\lambda})$ coincide.  But if $\Gal({\rm M}/{\rm L})$ had maximal 2-rank ($2^{k-1}$), these fields would be distinct.  Thus $\Gal({\rm M}/{\rm L})$ must have 2-rank $<2^{k-1}$.
\end{proof}

\section{Galois Groups and Computations} \label{comp}

There is a criterion of Jordan that, when satisfied, allows one to conclude a Galois group is large.  Let $f \in \Q[x]$ and define the \emph{Newton Index} $\mathcal{N}_f$ of $f$ to be the least common multiple of the denominators of all slopes of all $\NP_p(f)$ as $p$ ranges over all the primes \cite{farshid}.

\begin{thm}[Jordan's Criterion, \cite{jordan}]
Let $f = \Q[x]$ be irreducible of degree $n \geq 8$.  Suppose that $\mathcal{N}_f$ has a prime divisor $\ell$ in the range $n/2 < \ell < n-2$.  Then $\Gal(f)$ contains $A_n$. 
\end{thm}

The idea behind the criterion is that a transitive subgroup of $S_n$ containing an element of prime order  $\ell$  in the range $n/2 < \ell <n-2$ must contain $A_n$.  It is then a matter of checking if the discriminant is a square to determine if the Galois group is $A_n$ or $S_n$.   

In practice, one computes a number of Newton polygons and looks for slopes with a prime denominator in the ``Jordan range''.  In previous work on  the Kaneko-Zagier Polynomials \cite{rylan}, the Legendre Polynomials \cite{ch},  and certain Generalized Laguerre Polynomials \cite{verity}, we were able to implement Jordan's Criterion easily.  That was due to the nature of the coefficients of the polynomials. 

For example, in \cite{ch} we studied the polynomials
\[
\mathscr{J}_n^{\pm}(x) \ddef \sum_{j=0}^n \binom{n}{j} \far{2j \pm 1}{n}\, x^j = \sum_{j=0}^n c_jx^j
\]
where $\far{\alpha}{n} = (\alpha+2)(\alpha + 4) \cdots (\alpha +2n)$.  The prime divisors of the coefficients $c_j$ can roughly be divided into two parts:
\[
c_j = \underbrace{\binom{n}{j}}_{\text{divisible by $p \leq n$}} \cdot \underbrace{\far{2j\pm 1}{n}}_{\text{divisible by $p \in [2j \pm 1,2j \pm 1 + 2n]$}}.
\]
By appealing to results and conjectures (\emph{e.g.}, the Hardy-Littlewood conjecture) on the distribution of primes in intervals of length $2n$, we put forward a strong conjecture that once $n$ is sufficiently large that there will be a $p$-adic Newton Polygon of the following shape:
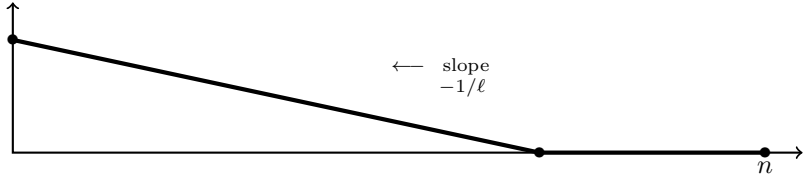
\begin{center}
\begin{tikzpicture}[scale=1,sizefont/.style={scale = 2}]
        \draw [<->,thick] (0,2) node (yaxis) [above] {}
        |- (10.5,0) node (xaxis) [right] {};
   \draw[ultra thick] (0,3/2)   --  (7,0); 
   \draw[ultra thick] (7,0)   --  (10,0); 
   \draw (6.5,1) node[left] {{ $\longleftarrow  \ \ \text{slope} \atop \ \ \ \ \   -1/\ell$}};
  \draw (0,3/2) node[left] {$$};
\draw (10,0) node {$\bullet$};
\draw (7,0) node {$\bullet$};
  \draw (10,0) node[below] {$n$};
\draw (0,3/2) node {$\bullet$};
\end{tikzpicture}
\end{center}

However, the coefficients of the Krawtchouck polynomials do not have such an agreeable form and do not lend themselves readily to Jordan's Criterion.  

Dedekind's Criterion, which says that the factorization of an integral polynomial modulo a good prime $p$ describes the cycle type of the Frobenius as a conjugacy class in the global Galois group, is an easy-to-implement numerical check that a Galois group is large: sample many primes and determine whether there is a prime-order element in the Jordan range.

\begin{exm}
Consider the polynomial $\K_{20}^{(20)}(x)$ and it's associated $\KK_{10}^{(0,20)}(x)$. Clearing denominators, one obtains an integral irreducible polynomial with discriminant 
\begin{align*}
&2^{28}\cdot3^{50}\cdot5^{33}\cdot7^8 \cdot 2857\cdot3371 \cdot \\
&3080247982713573950046529683277689810503273830007221192065657784224004955821.
\end{align*}
There are 410 primes in the interval $7 < p < 2857$.  Of these, 65 give a factorization indicating that the $\Gal_\Q (\KK_{10}^{(0,20)}(x))$ has an element of order 7, hence contains $A_{10}$.  Since the discriminant is not a rational square, we have $\Gal_\Q (\KK_{10}^{(0,20)}(x)) \simeq S_{10}$.
\end{exm}

The issue at hand for the general Krawtchouck polynomial is that none of this behavior is visible from the point of view of the Newton Polygon. In particular, the leading and constant coefficients of $\K_n^{(t)}(x)$ are
\[
\frac{(-2)^n}{n!} \text{ and } \binom{t}{n},
\]
respectively, indicating that small primes (those dividing $n$) are the likely candidates for interesting Newton Polygons, but these primes are also likely to divide the discriminant of the polynomial.  In practice, one finds (roughly) that the only prime that gives rise to interesting Newton Polygons is $p=2$, regardless of specialization $t \in \Q$.

In \textsf{Magma} \cite{magma} we performed the following experiment:  For all $\delta \in \lbrace 0,1 \rbrace$, $a \in [-10^4,10^4] \cap \Z$, $b \in [1,10^4] \cap \Z$, and $n=1,2,\dots,20$, whenever the polynomial $\KK_n^{(\delta,a/b)}(x)$ was irreducible, we computed its Galois group to be $S_n$.  Based on this evidence, we make the following conjecture. 

\begin{conjecture}
For all specializations $t \in \Q$ at which $\KK_n^{(\delta,t)}(x)$ is irreducible of degree $n$, the Galois group $\Gal_\Q(\KK_n^{(\delta,t)}(x)) \simeq S_n$.
\end{conjecture}

In particular, we conjecture that the Galois group never gets smaller than $A_n$ under specialization (hence the discriminant is never a rational square at irreducible specializations).  We conclude with the special case $n=3$ to explore this idea.

\begin{exm}
Consider the polynomials $\KK_3^{(\delta,t)}(x)$.  Shifting to remove the trace term and dividing by the leading coefficient, we obtain the following: 
\begin{align*}
\KK_3^{(0,t)}(x) & \sim x^3 + \left(\frac{-15}{8}t^2 + \frac{95}{8}t - \frac{131}{6} \right)x + \left( \frac{-5}{8}t^3 + \frac{55}{8}t^2 -\frac{325}{12}t +\frac{965}{27}\right) \\
\KK_3^{(1,t)}(x) & \sim x^3 + \left(\frac{-21}{8}t^2 + \frac{175}{8}t - \frac{637}{12} \right)x + \left(\frac{-7}{8}t^3 + \frac{105}{8}t^2 - \frac{833}{12}t + \frac{3305}{27} \right).
\end{align*}
The reducible specializations of these polynomials are parameterized by the rational points on the   curves defined by these polynomials, by \cite{hw}.  These curves are elliptic curves and can be transformed into Weierstrass form with explicit expressions
\begin{align*}
\KK_3^{(0,t)}(x) &= 0 \qquad \rightarrow \qquad E_0:=y^2 + xy + y = x^3 - x^2 - 62705x + 5793697 \\
\KK_3^{(1,t)}(x) &=0 \qquad \rightarrow  \qquad E_1:=y^2 + xy + y = x^3 - x^2 - 722882x + 185853889.
\end{align*}
In \textsf{Magma}, one computes that $E_0(\Q) \simeq E_1(\Q) \simeq \Z \times \Z \times \Z \times \Z$, hence each has infinitely many reducible specializations. 

If $t\in \Q$ is such that $\KK_3^{(\delta,t)}(x)$ is irreducible, then we check to see if the associated discriminants are squares to determine if the Galois group is $A_3$ or $S_3$.  Computing their discriminants and checking whether they are squares leads us to consider rational points on the genus-2 hyperelliptic curves
\begin{align*}
s^2&= \frac{3}{128} \left( 675t^6 - 11475t^5 + 81225t^4 - 303125t^3 + 622860t^2 - 668080t + 304704\right) \\
s^2&= \frac{3}{128} \left(2205t^6 - 50715t^5 + 492009t^4 - 2561825t^3 + 7539882t^2 - 11982460t + 8267304\right),
\end{align*}
when $\delta = 0,1$, respectively.  By Faltings' theorem, there are finitely many rational points on these curves, hence only finitely many specializations with Galois group $A_3$.  A search in \textsf{Magma} for rational points up to the bound $10^7$ yields the points:
\begin{align*}
\delta = 0:\qquad & (3,117/2),(4,165/4), (5,48), (6, 120), (14, 7680) \\
\delta =1: \qquad & (2,63/4),(3,21/2),(4,12),(5, 48), (12,3072), 
\end{align*}
all of which correspond to reducible specializations, giving evidence for the conjecture that all irreducible specializations have Galois group $S_3$. 
\end{exm}


\begin{thebibliography}{99}

\bibitem{rylan} J.~Cullinan, R.~Gajek-Leonard. On the Newton Polygons of Kaneko-Zagier lifts of supersingluar polynomials. Research in Number Theory \textbf{2} (1), 1 -- 16 (2016)

\bibitem{ch} J.~Cullinan, F.~Hajir. On the Galois groups of Legendre polynomials. Indag. Math. (N.S.) \textbf{25} (2014), no. 3, 534 -- 552.


\bibitem{kz} J.~Cullinan, F.~Hajir. Algebraic properties of Kaneko-Zagier lifts of supersingular polynomials.  To appear in \emph{Proceedings of the American Mathematical Society}.

\bibitem{verity} J.~Cullinan, N.~Scheel. On the arithmetic of Pad\'e approximants to the exponential polynomials. Journal of the Ramanujan Mathematical Society \textbf{37} (3) 1 -- 13 (2022)

\bibitem{filaseta} M.~Filaseta, T.~Kidd, O.~Trifonov. Laguerre polynomials with Galois group $A_m$ for each $m$. Journal of Number Theory \textbf{132} (2012), 776 -- 805.

\bibitem{grosswald} E.~Grosswald. On a simple property of the derivatives of Legendre's polynomials. Proceedings of the American Mathematical Society vol. 1 (1950), 553 -- 554.

\bibitem{farshid} F.~Hajir. Algebraic properties of a family of generalized Laguerre polynomials. Canad. J. Math. \textbf{61} (2009), no. 3, 583-603.

\bibitem{hw} F.~Hajir, S.~Wong. Specializations of one-parameter
families of polynomials. \textsl{Annales de L'Institut Fourier.}
\textbf{56}, 1127-1163, 2006

\bibitem{jordan} C.~Jordan. Sur la limite de transitivit\'e des groupes non altern\'es. Bull. Soc. Math. France. \textbf{1} (1872-3), 40-71.

\bibitem{kzorig} M.~Kaneko, D.~Zagier. Supersingular $j$-invariants, hypergeometric series, and Atkin's orthogonal polynomials. Computational perspectives on number theory (Chicago, IL, 1995), 97--126, AMS/IP Stud. Adv. Math., 7, Amer. Math. Soc., Providence, RI, 1998.

\bibitem{magma} W.~Bosma, J.~Cannon, C.~Playoust, The Magma algebra system. I. The user language, J. Symbolic Comput., 24 (1997), 235?265.


\bibitem{szego} G.~Szego. G.~Szeg\"{o}, Orthogonal Polynomials, 4th ed. Providence, 
RI: Amer. Math. Soc., 1975.

\bibitem{van_lint} J.H.~Van Lint. Introduction to Coding Theory, 3d Ed. Springer-Verlag, 1999.

\bibitem{wahab} J.H.~Wahab. New cases of irreducibility for Legendre polynomials. Duke Math. J. \textbf{19}, (1952) 165-176.

\end{thebibliography}
\end{document}